\documentclass[11pt]{amsart}
\usepackage{amssymb,amsfonts,amsthm,amsmath,mathrsfs,xspace,hyperref}
\usepackage{graphicx}
\usepackage[francais,english]{babel}
\usepackage{inputenc}
\usepackage[T1]{fontenc}
\usepackage[centering]{geometry}
\usepackage{csquotes}
\usepackage{color}
\usepackage{enumerate}
\usepackage{mathabx}
\usepackage{stmaryrd}
\usepackage{enumitem}

  \newcommand{\Z}{\ensuremath{\mathbb{Z}}}%
	\newcommand{\Q}{\ensuremath{\mathbb{Q}}}%
                \renewcommand{\P}{\ensuremath{\mathcal{P}}}%

\theoremstyle{definition}
  \newtheorem{defin}{Definition}[section]

	 \newtheorem{notation}[defin]{Notation}

\theoremstyle{plain}
  \newtheorem{thm}[defin]{Theorem}
  \newtheorem*{thm-other}{Theorem}
  \newtheorem{prop}[defin]{Proposition}
    \newtheorem{prop-def}[defin]{Proposition-Definition}

  \newtheorem{cor}[defin]{Corollary}
   \newtheorem{lem}[defin]{Lemma}

\theoremstyle{remark}
  \newtheorem{remark}[defin]{Remark}

\begin{document}

  \date{August 29, 2023}

\title{Some torsion-free solvable groups with few subquotients}

\author{Adrien Le Boudec}
\address{CNRS, UMPA - ENS Lyon, 46 all\'ee d'Italie, 69364 Lyon, France}
\email{adrien.le-boudec@ens-lyon.fr}

\author{Nicol\'as Matte Bon}
\address{
	CNRS,
	Institut Camille Jordan (ICJ, UMR CNRS 5208),
	Universit\'e de Lyon,
	43 blvd.\ du 11 novembre 1918,	69622 Villeurbanne,	France}

\email{mattebon@math.univ-lyon1.fr}

\thanks{Supported by the LABEX MILYON (ANR-10-LABX-0070) of Universit\'e de Lyon, within the program "Investissements d'Avenir" (ANR-11-IDEX-0007) operated by the French National Research Agency.}

\maketitle

\begin{abstract}
We construct finitely generated torsion-free solvable groups $G$ that have infinite rank, but such that all finitely generated torsion-free metabelian subquotients of $G$ are virtually abelian. In particular all finitely generated metabelian subgroups of $G$ are virtually abelian. The existence of such groups shows that there is no ``torsion-free version'' of P. Kropholler's theorem, which characterises solvable groups of infinite rank via their metabelian subquotients.
\end{abstract}

\section{Introduction}

A solvable group $G$ has \textbf{finite rank} if there is $k \geq 1$ such that all finitely generated (fg) subgroups of $G$ are generated by at most $k$ elements.  The class of groups of finite rank is stable under the operations of taking subgroups and quotients, and hence under taking subquotients. Recall that a group $Q$ is a \textbf{subquotient} of $G$ if there are subgroups $K,H$ of $G$ with $K \lhd H$ and $Q \simeq H/K$ (the term \enquote{section} is also used, especially in the literature on solvable groups). The simplest example of a fg solvable group of infinite rank is the wreath product $C_p \wr \Z$. The following celebrated theorem of P. Kropholler asserts that looking at the metabelian subquotients of a  fg solvable group $G$ suffices to detect that $G$ has infinite rank.

\begin{thm-other}[\cite{Kropholler84}]
	If $G$ is a fg solvable group of infinite rank, then $G$ admits a subquotient isomorphic to $C_p \wr \Z$ for some prime $p$. 
\end{thm-other}

One of the motivating questions of the present article is whether there exists a \enquote{torsion-free version} of this theorem. More precisely, if $G$ is a  fg  torsion-free solvable group of infinite rank, does $G$ always admit a fg torsion-free metabelian subquotient of infinite rank ? We answer this question in the negative:


\begin{thm} \label{thm-intro}
	The group $G$ below is a fg $3$-step solvable group with the following properties:
	\begin{enumerate}
		\item $G$ is torsion-free;
		\item $G$ contains a normal subgroup that is free abelian of infinite rank;
		\item \label{item-secti-virtab} every finitely generated torsion-free metabelian subquotient of $G$ is virtually abelian. 
	\end{enumerate}
\end{thm}

Even before considering subquotients, if $G$ is a  fg $n$-step solvable group ($n\geq 3$) with a given property $\P$, it is natural to ask whether $G$ admits  fg  metabelian \textit{subgroups} that retain $\P$. Natural examples are provided by  \enquote{$\P = $ not being virtually abelian}, or  \enquote{$\P = $ not being virtually nilpotent}. A positive answer to that kind of questions generally allows to reduce the study of certain problems about fg solvable groups to fg  metabelian groups, which are generally much more tractable. When $\P$ is the property of having infinite rank, the following consequence of Theorem \ref{thm-intro} shows that the answer is negative. The existence of groups with this property is also new. 

\begin{cor}
The group $G$ is a fg torsion-free solvable group of infinite rank such that every fg  metabelian subgroup of $G$ is virtually abelian.
\end{cor}

Recall that a solvable group has \textbf{finite torsion-free rank} if it admits a series with abelian quotients $G_i/G_{i+1}$ satisfying $\dim_\Q (G_i/G_{i+1} \otimes \Q) < \infty$. Another natural  question that arises from P. Kropholler's theorem is to ask whether every fg  solvable group of infinite torsion-free rank admits fg metabelian subquotients of infinite torsion-free rank.  This was answered in the negative by P. Kropholler, who constructed counter-examples in \cite{Kropholler85}. Despite having infinite torsion-free rank, these examples admit rather large subgroups containing torsion elements. The study of the class of solvable groups admitting no metabelian subquotient of infinite torsion-free rank is the subject of the recent work \cite{JacoboniKropholler} of Jacoboni--Kropholler. Theorem B in \cite{JacoboniKropholler} shows these groups enjoy strong structural restrictions.

\textbf{On the construction}. Our  family of groups from Theorem \ref{thm-intro}  is very much inspired by the construction of solvable groups of P. Kropholler \cite{Kropholler85}, Brieussel \cite{Brieu} and Brieussel--Zheng \cite{Brieu-Zhe}. The constructions from \cite{Brieu, Brieu-Zhe} take as inputs several parameters, including a sequence of groups $\Delta_n=\langle s_n, t_n\rangle$ coming with a specified pair of generators; and produce solvable groups $G$ such that on the one hand even if the groups $\Delta_n$ are very small (for instance finite or virtually abelian), the group $G$ might be much larger; and on the other hand good choices for the $\Delta_n$ allow some control on the algebraic structure of $G$. It was already observed in \cite{JacoboniKropholler} that a construction close to  \cite{Brieu} yields solvable groups of infinite torsion-free rank which do not admit metabelian subquotients of infinite torsion-free rank (like the groups in \cite{Kropholler85}). The construction has the property that in order to avoid metabelian subquotients of infinite torsion-free rank, it is necessary that the groups $\Delta_n$  have torsion elements, and as a consequence the groups $G$ from \cite{Brieu} and \cite{JacoboniKropholler} have torsion elements. The delicate aspect of our construction is to ensure torsion-freeness of the group $G$ while keeping the other properties intact.

\textbf{Acknowledgments}. We are very grateful to Peter Kropholler for decisive discussions regarding this problem, and for his comments on a preliminary version of this work. We thank J\'er\'emie Brieussel and Tianyi Zheng for a conversation on the return probability  on the groups constructed here. 

\section{Proofs}
 
\subsection{Torsion-free metabelian subquotients}

The purpose of this paragraph is to prove Proposition \ref{prop-fcext-tfmetab} below, which will be  used to check item \ref{item-secti-virtab} in Theorem \ref{thm-intro}. That proposition follows from standard arguments. We include proofs for completeness. 

Recall that the FC-center $FC(G)$ of a group $G$ is the subgroup of $G$ consisting of elements having a finite conjugacy class. A subgroup of $G$ is termed FC-central if it is contained in the FC-center of $G$. Recall also that a group is locally finite if all finitely generated subgroups are finite. 

In the sequel we denote by $\mathcal{L}$ the class of groups $G$ that admit a normal subgroup $L$ such that $L$ is locally finite and abelian, and $G/L$ is cyclic. We note that the class  $\mathcal{L}$ is stable under passing to subgroups and under taking quotients.

\begin{lem} \label{lem-FC-locfini-cyclic}
	Let $G$ be a finitely generated group that is torsion-free. If $G$ admits a finitely generated abelian normal subgroup $A$ such that $A$ is FC-central in $G$ and $G/A$ is in $\mathcal{L}$,  then $G$ is virtually abelian.
\end{lem}

\begin{proof}
	Since $A$ is finitely generated, its centralizer $C_G(A)$ has finite index in $G$. Hence upon replacing $G$ by a finite index subgroup, we can assume that $A$ is central in $G$. Let $N$ be a normal subgroup of $G$ containing $A$ and $t \in G$ such that $G = N \rtimes \left\langle t\right\rangle $ and $N/A$ is abelian. Since $A$ is central in $G$, $N$ is nilpotent. Since $N$ is torsion-free, the quotient $N/Z(N)$ of $N$ by its center is also torsion-free \cite[5.2.19]{Rob-course}. But here $N/Z(N)$ must be locally finite, so it follows that $N/Z(N)$ is trivial. So we deduce that $N$ is abelian. Now given $x \in N$, there exists $k \geq 1$ such that $x^k$ lies in the center of $G$. Hence $x^k = t x^k t^{-1} = (t x t^{-1})^k$. But $x$ and $t x t^{-1}$ commute since $N$ is abelian, so $[t,x]^k=1$. Since $G$ is torsion-free, we infer that $x$ commutes with $t$. Since $x$ was arbitrary in $N$ and $G = N \rtimes \left\langle t\right\rangle $, this shows that $G$ is abelian.
\end{proof}

We also denote by $\mathcal{C}$ the class of groups $G$ such that $G/FC(G)$ belongs to $\mathcal{L}$. 

\begin{prop} \label{prop-fcext-tf}
	Let $G$ be a finitely generated group that belongs to $\mathcal{C}$. If $G$ is torsion-free and metabelian, then $G$ is virtually abelian. 
\end{prop}

\begin{proof}
	Since fg metabelian groups satisfy the maximal condition on normal subgroups \cite{Hall}, the subgroup $FC(G)$ is finitely generated. The group $G$ being torsion-free, $FC(G)$ is also abelian \cite[14.5.9]{Rob-course}. Hence it follows that $FC(G)$ is a finitely generated abelian group, and Lemma \ref{lem-FC-locfini-cyclic} then implies that $G$ is virtually abelian.
\end{proof}

We deduce the following. 

\begin{prop} \label{prop-fcext-tfmetab}
	If $G$ is a group in $\mathcal{C}$, then every finitely generated torsion-free metabelian subquotient of $G$ is virtually abelian. 
\end{prop}

\begin{proof}
	Let $K$ be a fg torsion-free metabelian subquotient of $G$. The class $\mathcal{C}$ being stable under taking subgroups and quotients, $K$ belongs to $\mathcal{C}$. The statement then follows from  Proposition \ref{prop-fcext-tf}.
\end{proof}

\subsection{Notation}
In the sequel we denote by $C_k = \Z / k \Z = \left\lbrace 0, \ldots, k-1 \right\rbrace $. Let $D = \left\langle a, b \, | \, a^2=b^2=1 \right\rangle $ be the infinite dihedral group, and $Z$ the infinite cyclic subgroup of $D$ generated by $ab$, so that $D = Z \rtimes \langle a \rangle = Z \rtimes \langle b \rangle$. We denote by $D'$ the derived subgroup of $D$, which is the index two subgroup of $Z$ generated by $(ab)^2$.

\subsection{The groups}

Choose two strictly increasing sequences $(d_n)$ and $(k_n)$ of positive integers such that $d_1 > 1$ and $k_n \geq 2d_n$. For $n \geq 1$, let $H_n=D^{k_n}\rtimes C_{k_n}$, where  the action of $C_{k_n}$ on $D^{k_n}$ is a cyclic permutation of the factors. 

Let $\Delta= \Z \times \prod_n H_n$. We denote by $\pi_0: \Delta\to \Z$ the projection from $\Delta$ to the first factor $\Z$, and for all $n \geq 1$ $\pi_n\colon  \Delta\to H_n$ is the projection from $\Delta$ to  $H_n$.

We define two elements $s, t\in \Delta$ as follows:

\begin{itemize}
	\item Let $\sigma_0$ be a generator of the first factor $\Z$ of $\Delta$. For $n \geq 1$, we denote by $\sigma_n$ the element of $H_n$ that belongs to the normal subgroup $D^{k_n}$ of $H_n$, and that is defined by $\sigma_n(1)=a$, $\sigma_n(d_n)=b$, and $\sigma_n(j)=id$ for every $j \in \left\lbrace 1, \ldots, k_n \right\rbrace$ such that $j \neq 1,d_n$. We denote by $s$ the element of $\Delta$ defined by $\pi_0(s) = \sigma_0$ and $\pi_n(s) = \sigma_n$ for all $n \geq 1$. 
	
	\item $t$  is the element of $\Delta$ defined by $\pi_0(t) = 0$, and $\pi_n(t)$ is the generator $1$ of $C_{k_n}$ for all $n \geq 1$. 
\end{itemize}

\begin{defin}
	We denote by $G$ the subgroup of $\Delta$ generated by $s$ and $t$. 
\end{defin}

The group $G$ depends on $(d_n)$ and $(k_n)$, but to simplify we omit $(d_n)$ and $(k_n)$ from the notation.

\begin{notation}
In the sequel for $i \in \Z$, we write $s_i=t^ist^{-i}$.
\end{notation}



\begin{lem} \label{lem-center}
	The element $s^2$ is central in $G$.
\end{lem}

\begin{proof}
$s^2$  belongs to the first factor $\Z$ of $\Delta$ since $\pi_n(s)$ has order $2$ for all $n \geq 1$.
\end{proof}

\begin{lem} \label{lem-basis-properties}
	For all $i \geq 0$ the following hold:
	\begin{enumerate}
		\item \label{item-first} $\pi_n([s,s_i])$ belongs to $(D')^{k_n}$ for all $n \geq 1$.
		\item  \label{item-second}  If $n$ is such that $d_n-1 = i$ then the $d_n$-coordinate of $\pi_n([s,s_{i}])$ is equal to $(ba)^2$, and all other coordinates of of $\pi_n([s,s_{i}])$ are trivial.
		\item \label{item-third}  for all $n$ such that $d_n-1 > i$, $\pi_n([s,s_i])$ is trivial.
	\end{enumerate}
\end{lem}

\begin{proof}
\ref{item-first} is clear since $\pi_n(s_i)$ belongs to $(D)^{k_n}$ for all $n \geq 1$. When $d_n-1 = i$, the first statement of \ref{item-second} is a simple computation, and the second statement follows from the fact that the supports of $\pi_n(s)$ and $\pi_n(s_i)$ intersect only at coordinate $d_n-1 = i$ in view of the inequality $k_n \geq 2d_n$. \ref{item-third} is the true for the same reason.
\end{proof}


\begin{prop} \label{p-FC-extension}
The subgroup $FC(G)$ is free abelian of infinite rank, and $Q = G/ FC(G)$ is isomorphic to $C_2\wr \Z$.
\end{prop}

\begin{proof}
Let $R=\bigoplus_{n\ge 1} (D')^{k_n}$. So $R$ is the subgroup of $ \prod_n H_n$ consisting of elements having all their coordinates in $(D')^{k_n}$, and only a finite number of these coordinates are non-trivial. Set $A = \langle \sigma_0^2 \rangle \oplus R$. We claim that $G$ contains $A$. Since $\pi_n(s) = \sigma_n$ has order two for all $n \geq 1$ and $\pi_0(s^2) = \sigma_0^2$, it is enough to see that $R$ lies inside $G$, or equivalently that $G$ contains $(D')^{k_n}$ for all $n\geq1$. Since $\pi_n(t)$ permutes transitively the factors of $D^{k_n}$, it is actually enough to see that $G$ contains the copy of $D'$ located at the $d_n$-coordinate in $D^{k_n}$. Arguing by induction on $n$, we see that this follows from Lemma \ref{lem-basis-properties}. 

We note that $A$ is normal in $G$, as it is actually normalized by the entire group $\Delta$. We shall now check that $A$ lies in $FC(G)$. Since $s^2$ in central in $G$ (Lemma \ref{lem-center}), it is enough to check that $(D')^{k_n}$ is FC-central in $G$ for all $n \geq 1$. For a fixed $n \geq 1$, the subgroup $G_n$ of $G$ consisting of elements $g$ such that $\pi_n(g) \in (D')^{k_n}$ is a finite index subgroup of $G$ since $D'$ has finite index in $D$ and $C_{k_n}$ is finite. Moreover $G_n$ commutes with $(D')^{k_n}$ because $D'$ is abelian, so $(D')^{k_n}$  indeed lies in the FC-center of $G$.

We shall now verify that $Q = G/A$ is isomorphic to $C_2\wr \Z$. If we denote by $\overline{s_i}$ the image of $s_i$ in $Q$, then $\overline{s_i}$ has order two for all $i$, and it follows from Lemma \ref{lem-basis-properties} that $[\overline{s},\overline{s_i}]$ is trivial for all $i$. This implies that the map from  $C_2\wr \Z$ to $Q$ that sends the Dirac function at the identity to $\overline{s}$  and the generator of $\Z$ to $t$ induces a surjective group homomorphism. Moreover we easily see that no $\overline{s_i}$ belongs to the subgroup generated by the $\overline{s_j}$ for $j < i$. This implies that $C_2\wr \Z \to Q$ is an isomorphism. And since $C_2\wr \Z$ has trivial FC-center, we also see that $A = FC(G)$.
\end{proof}

We deduce the following:

\begin{cor} \label{cor-Ghasnosubquotient}
Every finitely generated torsion-free metabelian subquotient of $G$ is virtually abelian. 
\end{cor}

\begin{proof}
This follows from Propositions \ref{prop-fcext-tfmetab} and \ref{p-FC-extension}.
\end{proof}

\subsection{Torsion-freeness}
The construction of the group $G$ above is a variation of Brieussel's construction in \cite{Brieu}, which involves three sequence $(\ell_n)$, $(d_n)$ and $(k_n)$, and the groups $H_n=D_{\ell_n}^{k_n} \rtimes C_{k_n}$, where $D_{\ell_n}$ are finite dihedral groups. This also leads to FC-central extensions of $C_2 \wr \Z$, but with the property that $FC(G)$ is locally finite (and hence torsion). With the hope to avoid the appearance of torsion, in our definition the finite dihedral groups have been replaced by the infinite dihedral group $D$, and the extra factor $\Z$ has been added in order to make the generator $s$ of infinite order.  But at this point torsion-freeness of the group $G$ is very much not guaranteed. Observe that the projection of $G$ to each factor $H_n$ contains many elements of finite order (for instance $\pi_n(s_i)\in H_n$ is an involution for every $n$ and every $i$). One might expect that by taking appropriate  commutators of such elements (so as to arrange in particular the projection to the factor $\Z$ to be trivial), it should be easy to find elements of finite order in $G$. And indeed for most choices of the sequences $(d_n), (k_n)$, the group $G$ will contain many elements of finite order. Perhaps surprisingly, it turns out that for the specific choice $d_n=n+1$, the group $G$ is torsion-free. The goal of this paragraph is to prove this assertion.

\begin{defin}
	We denote $\varphi_a\colon D\to C_2$ the epimorphism such that  $\varphi_a(a)=1$ and $\varphi_a(b)=0$ (here and below we use additive notation for $C_2=\{0, 1\}$, so that $1$ denotes its generator). Similarly we denote $\varphi_b\colon D\to C_2$ the epimorphism such that $\varphi_b(b)=1$ and $\varphi_b(a)=0$.
\end{defin} 

Thus $\varphi_a(g)$ is given by the parity of the number of $a$'s in any word representing $g$, and similarly for $\varphi_b(g)$. We will use the following easy lemma. 

\begin{lem}\label{l-dihedral-infinite-order}
	Let $\gamma \in D$ be such that $\varphi_a(\gamma)=\varphi_b(\gamma)=1$. Then $\gamma$ has infinite order. 
\end{lem}

\begin{proof}
	Let $\varphi\colon D \to C_2\times C_2$ be given by $\varphi=(\varphi_a, \varphi_b)$. Recall that an element  $g\in D$ has order 2 if and only if it is conjugate either to $a$ or to $b$, thus $\varphi(g)\in \{(1, 0), (0, 1)\}$. Thus the condition $\varphi(\gamma)=(1, 1)$ tells at once that  $\gamma$ is not trivial (since $\varphi(\gamma)$ is not-trivial) and not an involution, and therefore has infinite order.
\end{proof}

\begin{prop} \label{prop-torsionfreeness}
Suppose that $d_n=n+1$ for all $n \geq 1$. Then the group $G$ is torsion-free.
\end{prop}

\begin{proof}
We retain the above notation, so that $A = FC(G)$ and $Q = G/A$.  Given $g\in G$ non-trivial, we show that  $g$ has infinite order. Write $g$ as a word in $s^{\pm 1}, t^{\pm 1}$ and, by a standard  argument,   replace this word by an equivalent word in the free group generated by $s, t$ to write $g$ as
\[g=s_{i_1}^{\epsilon_1}\cdots s_{i_k}^{\epsilon_k} t^m,\]
with $i_j,\epsilon_i, m\in \Z$. Note that   an expression of $g$ of the form above  is far from being unique (in particular, there might be repetitions in the indexes $i_j$).
 
If $m\neq 0$, then the projection of $g$ to $Q$ has infinite order, and thus so does $g$. Therefore we can assume that $m=0$. Similarly we note that if $\sum_{j=1}^k \epsilon_j\neq 0$, then $\pi_0(g)\in \Z$ is non-trivial, and thus again $g$ has infinite order. Therefore we can suppose that $g$ has the form
\begin{equation} g=s_{i_1}^{\epsilon_1}\cdots s_{i_k}^{\epsilon_k}, \quad\quad   \epsilon_j\in \Z, \quad \sum_j \epsilon_j=0.\label{e-element-g} \end{equation}
 For every $\ell\in \Z$ let $\alpha(\ell)\in C_2$ be given by
\[\alpha(\ell)=\sum_{j\colon i_j=\ell} \epsilon_j \quad \text{(mod 2)}.\]
Note that the projection of $g$ to $G/A\simeq C_2 \wr \Z$ coincides with  the configuration $\alpha \in \oplus_\Z C_2$. If the latter is the trivial configuration, we have $g\in A$, and thus $g$ has infinite order by Proposition \ref{p-FC-extension}. Therefore we can assume that this is not the case, so that there exists some $\ell$ such that $\alpha(\ell)=1$. On the other hand by definition of $\alpha$ we have
\[\sum_{j\in \Z} \alpha(j)=\sum_{i=1}^k \epsilon_k= 0 \quad \text{ (mod 2)},\]
so we deduce that the set
 \[E=\{\ell\in \Z \colon \alpha(\ell) =1\}\]
is non-empty and has even cardinality. Thus $|E|\ge 2$. Let $m$ and $M$ be respectively the smallest and largest element of $E$. Upon conjugating $g$ by $t^{-m}$, we can suppose that $m=0$, while $M>0$. To show that $g$ has infinite order, we will show that the projection $\pi_M(g)\in H_M$ has infinite order. To this end note that for every $i\in \Z$ the element $\pi_M(s_i)$ belongs to the normal subgroup $D^{k_M}$ of $H_M$, and  is equal to the element $r_i\in  D^{k_M}$ given by
\begin{equation} r_i(x)=\left\{\begin{array}{lr}a & \text{ if } x=i+1  \quad (\text{mod }k_M)\\
b& \text{ if } x=M+i+1 \quad (\text{mod } k_M)\\
e & \text{ otherwise.}\end{array} \right.\label{e-ri}\end{equation}
From \eqref{e-element-g} it follows that $\pi_M(g)$ is equal to the element $f\in D^{k_M}$ given by 
\[f=r_{i_1}^{\epsilon_1}\cdots r_{i_k}^{\epsilon_k}.\]
We wish to show that the coordinate in $f(M+1)\in D$ is an element of infinite order of $D$ by applying  Lemma \ref{l-dihedral-infinite-order}.  To this end, note that  $\varphi_b(f(M+1))$ is determined by the parity of the number of $b$s that appear in the expression 
\[f(M+1)=r_{i_1}^{\epsilon_1}(M+1)\cdots r_{i_k}^{\epsilon_k}(M+1).\]
By \eqref{e-ri}, we have $r_j(M+1)=b$ if and only if $i_j\in k_M\Z$. Thus
\[\varphi_b(f(M+1))=\sum_{j\colon i_j\in k_M \Z} \epsilon_j=\sum_{\ell\in k_M \Z}\alpha(\ell)\quad  \text{ (mod 2).}\]
Now recall that $0$ and $M$ are respectively the minimum and maximum values of $\ell\in \Z$ such that $\alpha(\ell)=1$ (mod 2). Since $k_M>2M+1>M$, it follows that in the previous sum there is exactly one term equal to 1 (mod 2), namely $\alpha(0)$. Therefore $\varphi_b(f(M+1))=1$. 
A similar reasoning shows that 
\[\varphi_a(f(M+1))=\sum_{j\colon i_j\in k_M \Z+M} \epsilon_j=\sum_{\ell\in k_M \Z+M}\alpha(\ell)\quad  \text{ (mod 2),}\]
and again, in the last sum there is exactly one term equal to 1 (mod 2) namely, $\alpha(M)$. Thus $\varphi_a(f(M+1))=1$.
By Lemma \ref{l-dihedral-infinite-order}, $f(M+1)$ has infinite order. Therefore so does $f=\pi_M(g)$, and so does $g$. \qedhere
 \end{proof}

\subsection{Conclusion}

The proof of Theorem \ref{thm-intro} is now complete, as it follows from  Proposition \ref{prop-torsionfreeness}, Proposition  \ref{p-FC-extension} and Corollary \ref{cor-Ghasnosubquotient}.

\begin{remark}
By letting the sequence $(k_n)$ vary, one can obtain a continuum of isomorphism classes of groups satisfying Theorem \ref{thm-intro}. Here is a sketch of proof. The conjugation action of $G$ on its FC-center $A = FC(G)$ defines a linear representation on the $\Q$-vector space $V=A\otimes \Q$,  which, by the argument in Proposition \ref{p-FC-extension}, decomposes as a sum of finite dimensional representations $V=\Q\oplus( \bigoplus_n V_{n})$, with $V_{n}=( D')^{k_n} \otimes \Q$. Assuming that all $k_n$ are prime numbers, one verifies that the representation $V_n$ is irreducible. So by Schur's lemma it follows that $(k_n)$ is exactly the sequence of dimensions of irreducible finite dimensional subrepresentations of $V$ of dimension at least $2$. In particular the sequence $(k_n)$ can be reconstructed from $G$ abstractly. 
\end{remark}

\begin{remark}
The fact that the group $G$ has infinite rank is witnessed by the existence of a (sub)quotient isomorphic to $C_2\wr \Z$. It would be interesting to have a construction, for an arbitrary prime $p$, of a group $G$ as in Theorem \ref{thm-intro} which admits $C_p\wr \Z$ as a subquotient. We thank the referee for pointing out this question. 
\end{remark}

\subsection{Final comments}

Our initial motivation for wondering about the existence of (non-virtually abelian) fg torsion-free solvable groups with every fg  metabelian subgroup virtually abelian, grew out from the work \cite{LBMB-growth-solv}. There we prove that if $G$ is a fg torsion-free metabelian group, then the growth of every faithful action of $G$ (as defined in \cite{LBMB-growth-solv}) is at least quadratic, provided $G$ is not virtually abelian. We also conjecture that the same holds for torsion-free solvable groups. Since the desired property passes to overgroups, the metabelian case automatically implies that the conjecture holds for the solvable groups that contain a fg metabelian subgroup that is not virtually abelian. The examples presented in this note show that this does not include all torsion-free solvable groups. 

Finally we mention that another motivation for a better understanding of the class of torsion-free solvable groups comes from the study of random walks and isoperimetry.   Pittet and Saloff-Coste showed that on any torsion-free solvable group of finite rank which is not virtually nilpotent, the return probability of the simple random walk has the slowest possible decay for a group of exponential growth, namely $\exp(-n^{\frac{1}{3}})$  \cite{Pittet-Saloff-Coste}. They  conjectured that, conversely, this decay should characterise groups of finite rank among torsion-free solvable groups  \cite{Saloff-Coste-conj}. This conjecture is naturally related to the theory surrounding Kropholler's theorem and its possible generalisations, as the return probability does not decrease when passing to subquotients. But the groups constructed here seem unlikely to be counterexamples to the conjecture.  

\bibliographystyle{amsalpha}
\bibliography{bib-torsionfree}

\end{document}